\documentclass{amsart}
\usepackage{amsmath, amscd, amssymb, amsthm}
\usepackage{bbm}
\usepackage{latexsym}
\usepackage{amsfonts}
\usepackage{graphicx}
\usepackage{xcolor}
\definecolor{revieworange}{RGB}{224,112,0}
\definecolor{reviewred}{RGB}{180,35,35}
\definecolor{reviewblue}{RGB}{0,82,170}

\usepackage[all,cmtip]{xy}
\usepackage[colorlinks,linkcolor=blue,breaklinks=true,urlcolor=blue,citecolor=blue,anchorcolor=blue,pagebackref]{hyperref}%
\usepackage{geometry}
\newtheorem{theorem}{Theorem}
\newtheorem{lemma}{Lemma}
\newtheorem{corollary}[theorem]{Corollary}

\newtheorem{proposition}{Proposition}

\newtheorem{conjecture}{Conjecture}

\numberwithin{equation}{section}

\newcommand{\R}{{\mathbb R}}

\newcommand{\C}{{\mathbb C}}

\renewcommand*\backref[1]{}
\renewcommand*\backrefalt[4]{ \ifcase #1 \or (cited on page #2) \else (cited on pages #2) \fi}

\newcommand{\be}{\begin{equation}}
\newcommand{\ee}{\end{equation}}
\newcommand{\bea}{\begin{eqnarray}}
\newcommand{\eea}{\end{eqnarray}}

\newcommand{\tr}{\mathrm{tr}}
\newcommand{\im}{\mathrm{Im}}
\newcommand{\re}{\mathrm{Re}}

\newcommand{\vphi}{\varphi}

\newcommand{\om}{\omega}

\def\XXint#1#2#3{{\setbox0=\hbox{$#1{#2#3}{\int}$ }
\vcenter{\hbox{$#2#3$ }}\kern-.6\wd0}}

\newcommand{\nab}{\nabla^{b}}
\newcommand{\nac}{\nabla^{c}}
\newcommand{\Tb}{T^{b}}
\newcommand{\Rb}{R^{b}}

\newcommand{\End}{\operatorname{End}}

\newcommand{\Vol}{\operatorname{Vol}}

\begin{document}

\title[The Fino-Vezzoni conjecture on balanced BTP manifolds]{The Fino-Vezzoni conjecture on balanced Bismut torsion-parallel manifolds}

\author{Shuwen Chen}
\address{Shuwen Chen. School of Mathematical Sciences, Chongqing Normal University, Chongqing 401331, China}
\email{{3153017458@qq.com}}\thanks{Zheng is the corresponding author. He is partially supported by National Natural Science Foundations of China
with the grant No. 12471039 and  12141101, and is supported by the 111 Project D21024.}

\author{Fangyang Zheng}
\address{Fangyang Zheng. School of Mathematical Sciences, Chongqing Normal University, Chongqing 401331, China}
\email{20190045@cqnu.edu.cn; franciszheng@yahoo.com} \thanks{}

\subjclass[2020]{53C55 (primary), 53C05 (secondary)}
\keywords{Hermitian manifolds; balanced metrics; pluriclosed metrics; Bismut connection; parallel torsion; Fino--Vezzoni conjecture}

\begin{abstract}
We study the Fino-Vezzoni conjecture on compact complex manifolds carrying a balanced Bismut torsion-parallel Hermitian metric. We prove that if such a manifold also admits a pluriclosed Hermitian metric, then it admits a K\"ahler metric.
\end{abstract}

\maketitle

\markleft{Shuwen Chen and Fangyang Zheng}
\markright{The Fino-Vezzoni conjecture on balanced BTP manifolds}

\tableofcontents

\section{Introduction}\label{sec:intro}

Two Hermitian conditions that occur frequently in non-K\"ahler geometry are the balanced and pluriclosed conditions. Let $(M^n,g)$ be a Hermitian manifold with fundamental form $\om$. We call $g$ {\em balanced} if $d(\om^{n-1})=0$, and {\em pluriclosed} if $\partial\bar\partial\om=0$. Fino and Vezzoni proposed the following conjecture concerning the coexistence of these two types of metrics \cite{FV2,FV16}.

\begin{conjecture}[Fino-Vezzoni, {\cite{FV2,FV16}}]\label{conj:FV}
Let $M^n$ be a compact complex manifold. If $M$ admits a balanced Hermitian metric and a pluriclosed Hermitian metric, then $M$ admits a K\"ahler metric.
\end{conjecture}

When $n=2$, the balanced condition is simply $d\omega=0$, so the metric is K\"ahler. Hence the conjecture is relevant only in complex dimension at least $3$. Another point is that the two metrics in the hypothesis need not be the same. In fact, Alexandrov and Ivanov proved that a Hermitian metric which is both balanced and pluriclosed is K\"ahler \cite{AI}. Thus the problem concerns two, in general different, Hermitian metrics on the same compact complex manifold.

Several cases of the conjecture are known. Chiose proved it for manifolds in Fujiki class ${\mathcal C}$ \cite{Chiose}, and Verbitsky treated twistor spaces \cite{Verbitsky}. The results of Fu-Li-Yau and Fei give further examples among non-K\"ahler Calabi-Yau manifolds \cite{FuLiYau,Fei}, while Otiman established the conjecture for Oeljeklaus-Toma manifolds \cite{Otiman}. There is also a curvature version of this phenomenon: Zhao and Zheng showed that a compact non-K\"ahler manifold admitting a Bismut K\"ahler-like metric cannot admit a balanced metric \cite{ZhaoZhengPluriclosed}.

There has also been substantial progress in the invariant setting. In their earlier work, Fino and Vezzoni verified the conjecture for $6$-dimensional nilpotent groups and for $6$-dimensional solvable groups of Calabi--Yau type \cite{FV2}. They later proved it for
$2$-step nilpotent groups in any dimension \cite{FV16}, and the subsequent work of Arroyo and Nicolini implies that the conjecture holds for nilmanifolds in general \cite{ArroyoNicolini}. The conjecture is also known for compact semisimple Lie groups \cite{FGV,Podesta}. Giusti and Podest\`a treated compact quotients of even-dimensional non-compact real simple Lie groups of inner type, constructing invariant complex structures which admit balanced metrics but no pluriclosed metrics \cite{GPodesta}. Fino and Paradiso obtained further results for almost abelian Lie algebras and for a class of almost nilpotent solvable Lie algebras \cite{FP1,FP2,FP3}. Freibert and Swann established the conjecture for `pure-type' $2$-step solvable Lie algebras \cite{FSwann22,FSwann}; more recently, Fusi and Gentili extended this to all unimodular $2$-step solvable Lie algebras \cite{FusiGentili}.

For Lie algebras containing a $J$-invariant abelian ideal of codimension $2$, Li and Zheng proved the conjecture in \cite[Theorem 1.1]{LiZheng}. Cao and Zheng subsequently removed the $J$-invariance assumption and settled the codimension-$2$ abelian-ideal case \cite[Theorem 2 and Corollary 3]{CaoZheng}. More recently, Fino and Vezzoni obtained a result for the pluriclosed flow on compact quotients of Lie groups with an invariant Chern--Ricci-flat balanced background and an invariant pluriclosed initial metric \cite{FinoVezzoniFlow}. Kwong proved the conjecture for compact discrete quotients of complex homogeneous spaces with compact isotropy, in particular for compact quotients of Lie groups \cite{KwongHomogeneous}.

In this paper, we consider Bismut torsion-parallel metrics. Given a Hermitian manifold $(M^n,g)$ with fundamental form $\omega$, let $\eta$ be the global $(1,0)$-form on $M^n$ determined by
\begin{equation}\label{eq:Gauduchon-formula}
\partial(\om^{n-1})=-\eta\wedge\om^{n-1}.
\end{equation}
It is the {\em Gauduchon torsion $1$-form} \cite{Gauduchon84}; in particular, $g$ is balanced exactly when $\eta=0$. Let $\nab$ be the Bismut connection of $g$ and $\Tb$ its torsion. The {\em Bismut connection} is the unique Hermitian connection with totally skew-symmetric torsion \cite{Bismut}. Following Zhao and Zheng \cite{ZhaoZ22}, we call $g$ {\em Bismut torsion-parallel}, or {\em BTP}, if $\nab\Tb=0$. Their curvature characterization also gives the equivalent condition $\nab T=0$, where $T$ is the Chern torsion. Balanced non-K\"ahler BTP metrics form a restrictive class, although in higher dimensions their geometry is considerably less rigid than in complex dimension $3$.

We shall also use the generalized Gauduchon condition introduced by Fu, Wang, and Wu. For $1\leq k\leq n-1$, a Hermitian metric with fundamental form $\om$ is called {\em $k$-Gauduchon} if
$$
\partial\bar\partial(\om^k)\wedge\om^{n-k-1}=0;
$$
see Fu-Wang-Wu \cite[Definition 1.1]{FuWangWu}. The case $k=n-1$ is the usual Gauduchon condition. When $n\geq3$, the case $k=n-2$ will be used in the last step of the proof. Note that this case means that $\partial\bar\partial(\om^{n-2})\wedge\om =0$, which is weaker than $\partial\bar\partial(\om^{n-2})=0$. The latter is called {\em astheno-K\"ahler} and was introduced by Jost and Yau in \cite{JostYau}.

For BTP threefolds, Zhao and Zheng first proved that the
Fino--Vezzoni conjecture holds in the non-balanced case
\cite[Corollary 1.19]{ZhaoZ22}. More recently, they classified
compact balanced BTP threefolds \cite[Theorem 1.1]{ZZThreefold} and obtained the conjecture in the balanced case as well \cite[Corollary 5.8]{ZZThreefold}. Consequently, the Fino--Vezzoni conjecture is known for all compact BTP threefolds. The proof given below is independent of the three-dimensional classification and works uniformly for all
$n\geq3$.

\begin{theorem}\label{thm:main}
Let $M^n$ be a compact complex manifold. Suppose that $M$ admits a balanced Bismut torsion-parallel Hermitian metric $g$ and a pluriclosed Hermitian metric $h$. Then $M$ admits a K\"ahler metric.
\end{theorem}

Thus the Fino-Vezzoni conjecture holds on compact balanced BTP manifolds in every complex dimension.

We give a short description of the argument. Since the Chern torsion is $\nab$-parallel, its value at each point defines a complex Lie bracket which is preserved by Bismut parallel transport. A global $\nab$-parallel $(1,0)$-vector field then gives, through the Hermitian part of the corresponding inner derivation, a trace-free parallel Hermitian derivation. The associated real $(1,1)$-form is closed. On the other hand, the pluriclosed metric is averaged with respect to the Bismut holonomy to produce a positive $\nab$-parallel Hermitian endomorphism $K$ with the integral pairing property needed later. We then maximize $\log\det$ on the affine space obtained from $K$ by the closed parallel variations above. At a maximizer $D$, the first variation gives a positive balanced form $\alpha_D$, and $\alpha_D-\alpha_K$ is closed. The pluriclosed condition, together with the averaging identity, shows that $\alpha_D$ is $(n-2)$-Gauduchon. The balanced and $(n-2)$-Gauduchon identities then reduce the remaining term to the Hodge--Riemann form of the primitive $(2,1)$-form $\partial\alpha_D$, which forces $\partial\alpha_D=0$. Thus $\alpha_D$ is K\"ahler.

A determinant-maximization argument also occurs in Kwong's recent homogeneous-space proof \cite{KwongHomogeneous}, but the finite-dimensional family used there comes from invariant Aeppli classes. Here it is obtained instead from Bismut holonomy and parallel Hermitian derivations; no transitive group action is assumed.

\vspace{0.3cm}

\section{Preliminaries}

We first set up the notation, following the conventions in \cite{YZ18Cur,ZhaoZ22}. Let $(M^n,g)$ be a Hermitian manifold of complex dimension $n$, and write $g=\langle\, ,\,\rangle$, extended bilinearly over $\C$. Denote by $T^{1,0}M$ the holomorphic tangent bundle. Let $e={}^t\!(e_1,\ldots,e_n)$ be a local unitary frame of $T^{1,0}M$, and let $\vphi={}^t\!(\vphi_1,\ldots,\vphi_n)$ be its dual coframe, so that $\vphi_i(e_j)=\delta_{ij}$. Unless otherwise stated, all frame calculations below are understood with respect to such a unitary frame.

Denote by $\nac$ and $\nab$ the Chern and Bismut connections, respectively. We write $T^c=T$ and $R^c$ for the torsion and curvature of $\nac$, and $\Tb$ and $\Rb$ for those of $\nab$. The curvature tensor $R^D$ of a linear connection $D$ on a Riemannian manifold $M$ is defined  by
$$ R^D(x,y,z,w) = \langle R^D_{xy}z, \ w \rangle = \langle D_xD_yz-D_yD_xz - D_{[x,y]}z, \ w \rangle, $$
where $x,y,z,w$ are tangent vectors fields on $M$. In particular,
$$
R^b_{i\bar j k\bar\ell}
=
\left\langle R^b(e_i,\bar e_j)e_k,\bar e_\ell\right\rangle.
$$
The Chern torsion components are defined by
\begin{equation}\label{eq:Chern-torsion-components}
T^c(e_i,\bar e_j)=0,
\qquad
T^c(e_i,e_j)=\sum_kT^k_{ij}e_k,
\qquad
T^k_{ij}=-T^k_{ji}.
\end{equation}
Write
$$
\om=\sqrt{-1}\sum_i\vphi_i\wedge\bar\vphi_i,
\qquad
dV_g=\frac{\om^n}{n!}.
$$
The Gauduchon torsion 1-form is given by
$$
\eta=\sum_j\eta_j\vphi_j,
\qquad
\eta_j=\sum_iT^i_{ij}.
$$
\begin{equation}\label{eq:balanced}
 g\text{ is balanced}
 \quad\Longleftrightarrow\quad
 \sum_iT^i_{ij}=0,
 \qquad 1\leq j\leq n.
\end{equation}
For the Chern connection, write $\theta^c=\theta$ and $\Theta^c=\Theta$ for the connection and curvature matrices under $e$, and let $\tau={}^t\!(\tau_1,\ldots,\tau_n)$ be the column vector of Chern torsion $2$-forms. The first two Chern structure equations are
\begin{equation}\label{eq:Chern-structure-equations}
d\vphi=-{}^t\!\theta\wedge\vphi+\tau,
\qquad
d\theta=\theta\wedge\theta+\Theta,
\end{equation}
where $\tau_k=\frac12\sum_{i,j}T^k_{ij}\vphi_i\wedge\vphi_j.$ We use $\theta^b$ and $\Theta^b$ for the Bismut connection and curvature matrices. Let $\gamma = \nabla^b-\nabla^c$ be the tensor, and for simplicity we will also write $\gamma = \theta^b -\theta$ under $e$. Then by \cite{YZ18Cur} we have
\begin{equation} \label{eq:gamma}
 \gamma e_i = \sum_j \gamma_{ij} e_j = \sum_{j,k} \big( T^j_{ik}\varphi_k - \overline{T^i_{jk} } \, \overline{\varphi}_k \big) e_j.
\end{equation}
Write
$$ B_{i\bar{j}}=\sum_{r,s} T^j_{rs} \overline{ T^i_{rs} } , \ \ \ \ A_{i\bar{j}} =\sum_{r,s} T^r_{is} \overline{T^r_{js}}. $$
The tensors $A$ and $B$ are globally defined Hermitian
$(1,1)$-tensors. After raising one index with $g$, we also
regard them as Hermitian endomorphisms of $T^{1,0}M$.

\begin{lemma}[Zhao--Zheng, {\cite[Theorem 1.1 and Proposition 1.5]{ZhaoZ22}}]\label{prop:known-BTP}
Let $(M^n,g)$ be a BTP Hermitian manifold. Then $\nab T=0$, and under any local unitary frame $e$ the following identities hold:
\begin{eqnarray}
&& \ \ \   R^b_{ijk\bar{\ell}}=0,
\qquad
R^b_{i\bar{j}k\bar{\ell}}=R^b_{k\bar{\ell}i\bar{j}},
\qquad
\forall\,1\leq i,j,k,\ell\leq n,
\label{eq:pair-symmetry}
\\
&&  \ \ \  \sum_r\bigl(T^\ell_{ir}T^r_{jk}+
T^\ell_{jr}T^r_{ki}+T^\ell_{kr}T^r_{ij}\bigr)=0,
\qquad
\forall\,1\leq i,j,k,\ell\leq n,
\label{eq:torsion-Jacobi}
\\
&&  \ \ \ R^b_{i\bar{j}k\bar{\ell}}-R^b_{k\bar{j}i\bar{\ell}}=\sum_r\bigl(
T^\ell_{ir}\overline{T^k_{jr}}+T^j_{kr}\overline{T^i_{\ell r}}-T^r_{ik}\overline{T^r_{j\ell}}-T^j_{ir}\overline{T^k_{\ell r}}-T^\ell_{kr}\overline{T^i_{jr}}\bigr),
\qquad
\forall\,1\leq i,j,k,\ell\leq n,
\label{eq:curvature-torsion}
\\
&& \ \ \  \nabla^b A=\nabla^bB=0,\, \  [A,B]=0.
\label{eq:AB-parallel-commute}
\end{eqnarray}
\end{lemma}

\medskip
\noindent\textbf{The pointwise Chern-torsion bracket.}
Fix $p\in M$ and put $V:=T_p^{1,0}M$. The Hermitian metric induced by $g$ on $V$ is written as $(x,y):=\langle x,\bar y\rangle$. Since $T$ is a tensor, its value at $p$ defines a skew-symmetric complex-bilinear map $T_p:\Lambda^2V\to V$.
By Ni \cite[Theorem 4.1]{NiHolonomy}, the rule
\begin{equation*}
[x,y]_c:=T_p(x,y),
\qquad
x,y\in V,
\end{equation*}
defines a complex Lie algebra structure on $V$. Since $\nab T=0$ and $\nab g=0$, Bismut parallel transport identifies these pointwise Hermitian Lie algebras by unitary Lie-algebra isomorphisms. The subscript $c$ stands for Chern.

For $x\in V$, let $L_x:V\to V$ be defined by $L_x(y)=[x,y]_c$. Thus $L_x=\operatorname{ad}_x$ is an inner derivation. We write $\operatorname{Der}(V,[\, ,\, ]_c)$ for the complex vector space of derivations, and $L_x^*$ for the adjoint with respect to the above Hermitian inner product. We also use
$$
\im T
=
\operatorname{span}_{\C}\{T(x,y):x,y\in V\}
=
[V,V]_c.
$$
For $Q\in\End(V)$ and a skew bilinear map $P:\Lambda^2V\to V$, define
\begin{equation}\label{eq:pi-action}
(\pi(Q)P)(x,y)
=
Q(P(x,y))-P(Qx,y)-P(x,Qy).
\end{equation}
Thus $Q$ is a derivation of $(V,[\, ,\, ]_c)$ exactly when $\pi(Q)T=0$. If $g$ is balanced, then \eqref{eq:balanced} gives $\tr L_x=0$ for every $x\in V$; equivalently, the pointwise Chern-torsion Lie algebra is unimodular.

Let $H:T^{1,0}M\rightarrow T^{1,0}M$ be a complex-linear
endomorphism satisfying $H^*=H$ and $\nabla^bH=0$, where the
adjoint is taken with respect to $g$. Extend $H$ to
$T^{0,1}M$ by complex conjugation, and denote by the same letter
the induced real endomorphism of $TM$. Then $HJ=JH$. We associate
with $H$ the real $2$-form
$$
\alpha_H(X,Y)=g(JHX,Y),\qquad X,Y\in TM.
$$
Indeed, since $H^*=H$, $J^*=-J$, and $HJ=JH$, one has
$(JH)^*=-JH$; hence $\alpha_H$ is skew-symmetric. Moreover,
$\alpha_H(JX,JY)=\alpha_H(X,Y)$, so $\alpha_H$ is of type
$(1,1)$. Equivalently, if $e={}^t\!(e_1,\ldots,e_n)$ is a local
$g$-unitary frame with dual coframe $\vphi$ and
$H_{i\bar j}=g(He_i,\bar e_j)$, then
$$
\alpha_H=\sqrt{-1}\sum_{i,j}H_{i\bar j}
\vphi_i\wedge\bar\vphi_j.
$$
The identity $H_{i\bar j}=\overline{H_{j\bar i}}$ implies
$\bar\alpha_H=\alpha_H$. If $H>0$, then
$g_H(X,Y)=g(HX,Y)$ is a Hermitian metric whose fundamental form is $\alpha_H$.

We first compute $\partial\alpha_H$ in terms of the Chern torsion of $g$. The resulting formula will be used to characterize the closedness of $\alpha_H$.

\begin{lemma}\label{lem:partial-alpha}
For any $x\in M$, choose a local unitary frame
$e={}^t\!(e_1,\ldots,e_n)$ near $x$ such that
$\theta^b|_x=0$ and $H$ is diagonal at $x$, namely
$He_i=h_ie_i$ with $h_i\in\R$. Then at $x$,
\begin{equation}\label{eq:partial-alpha}
\partial\alpha_H
=
\frac{\sqrt{-1}}2
\sum_{i,j,k}(h_i+h_j-h_k)T^k_{ij}\,
\vphi_i\wedge\vphi_j\wedge\bar\vphi_k.
\end{equation}
\end{lemma}

\begin{proof}
Fix $x\in M$ and choose the local unitary frame $e$ as above.
Since $\theta^b|_x=0$, \eqref{eq:gamma} gives
$\theta=-\gamma$ at $x$. Moreover, since $\nab H=0$, the
coefficients of $H$ have vanishing first derivatives at $x$
with respect to this frame.  By \eqref{eq:Chern-structure-equations}, the first Chern structure equation is $d\vphi_k=-\sum_j\theta_{jk}\wedge\vphi_j+\tau_k$, where $\tau_k=\frac12\sum_{i,j}T^k_{ij}\vphi_i\wedge\vphi_j$. Starting from $\alpha_H=\sqrt{-1}\sum_kh_k\vphi_k\wedge\bar\vphi_k$, we have
$$
\partial\alpha_H=\sqrt{-1}\sum_kh_k\big(\partial\vphi_k\wedge\bar\vphi_k-\vphi_k\wedge\partial\bar\vphi_k\big).
$$
Substitute $\theta=-\gamma$ and the expression for $\gamma$ in \eqref{eq:gamma}. The torsion term in $\partial\vphi_k$ contributes $-h_kT^k_{ij}$ to the coefficient of $\vphi_i\wedge\vphi_j\wedge\bar\vphi_k$, while the two connection terms obtained from differentiating the holomorphic input factors contribute $h_iT^k_{ij}$ and $h_jT^k_{ij}$. After using $T^k_{ij}=-T^k_{ji}$ and collecting the two ordered copies of each pair $(i,j)$, the coefficient becomes $\frac{\sqrt{-1}}2(h_i+h_j-h_k)T^k_{ij}$, which is \eqref{eq:partial-alpha}.
\end{proof}

Formula \eqref{eq:partial-alpha} immediately gives the following characterization.
\begin{lemma}\label{lem:closed-derivation}
Let $H=H^*$ be $\nab$-parallel. The following conditions are equivalent:
\begin{enumerate}
\item $d\alpha_H=0$;
\item $\partial\alpha_H=0$;
\item $H\in\operatorname{Der}(T^{1,0}_pM,[\, ,\, ]_c)$ at every $p\in M$, that is,
$$
H[X,Y]_c=[HX,Y]_c+[X,HY]_c
$$
for all $X,Y\in T_p^{1,0}M$.
\end{enumerate}
In particular, if $H>0$ and one of these conditions holds, then $\alpha_H$ is a K\"ahler form.
\end{lemma}

\begin{proof}
At a fixed point, diagonalize $H$ as in Lemma \ref{lem:partial-alpha}. The coefficient of $e_k$ in
$$
H[e_i,e_j]_c-[He_i,e_j]_c-[e_i,He_j]_c
$$
is $(h_k-h_i-h_j)T^k_{ij}$. Hence $H$ is a derivation if and only if $(h_i+h_j-h_k)T^k_{ij}=0$ for all $i,j,k$. By \eqref{eq:partial-alpha}, this is equivalent to $\partial\alpha_H=0$. Since $\alpha_H$ is real of type $(1,1)$, its exterior derivative decomposes as $d\alpha_H=\partial\alpha_H+\bar\partial\alpha_H$, with $\bar\partial\alpha_H=\overline{\partial\alpha_H}$. Thus $\partial\alpha_H=0$ if and only if $d\alpha_H=0$. If $H>0$, then $\alpha_H$ is positive, so closedness makes it a K\"ahler form.
\end{proof}

\vspace{0.3cm}

\section{Parallel vectors and closed Hermitian variations}

Throughout this section, $g$ denotes the balanced BTP metric fixed in Theorem \ref{thm:main}. We introduce the notation
$$
\mathcal P^b(M)
:=
\left\{
Z\in\Gamma(T^{1,0}M):
\nabla^b Z=0
\right\}
$$
for the complex vector space of global $(1,0)$-vector fields
parallel with respect to the Bismut connection of $g$. Since $\nabla^bJ=0$, the connection $\nabla^b$ is complex linear on $T^{1,0}M$, and hence $\mathcal P^b(M)$ is a complex vector space.

Since $M$ is connected, evaluation at any fixed point $p\in M$
identifies $\mathcal P^b(M)$ with the fixed-point subspace
$$
\left(T_p^{1,0}M\right)^{\operatorname{Hol}_p(\nabla^b)}
$$
of the full Bismut holonomy group; see
\cite[Chapter II]{KobayashiNomizuI}. In particular,
$\dim_{\C}\mathcal P^b(M)\leq n$. Moreover, every
$Z\in\mathcal P^b(M)$ has constant $g$-norm, since
$\nabla^bg=0$. Thus a nonzero element of $\mathcal P^b(M)$ has positive constant length.

For brevity, we write $\mathcal P=\mathcal P^b(M)$ throughout the remainder of the paper. For later use, we calculate the following consequence of the parallelness of $Z$ and the BTP curvature symmetry \eqref{eq:pair-symmetry}.

\begin{lemma}\label{lem:curvature-parallel-vector}
Let $Z\in{\mathcal P}$ and assume $|Z|=1$. Choose a local unitary frame with $e_n=Z$. Then
$$
R^b_{i\bar j n\bar\ell}
=
R^b_{i\bar j k\bar n}
=
R^b_{n\bar j k\bar\ell}
=
R^b_{i\bar n k\bar\ell}
=0
$$
for all $1\leq i,j,k,\ell\leq n$.
\end{lemma}

\begin{proof}
Since $\nab Z=0$, one has $R^b(X,Y)Z=0$, hence $R^b_{i\bar j n\bar\ell}=0$. Curvature endomorphisms of a Hermitian connection are skew-Hermitian, so $R^b_{i\bar j k\bar n}=0$. The pair symmetry in \eqref{eq:pair-symmetry} then gives the vanishing when the first or the second vector index is $n$. Conjugation gives the corresponding identities with barred indices.
\end{proof}

Combining Lemma \ref{lem:curvature-parallel-vector} with \eqref{eq:curvature-torsion} gives the following consequence for a Bismut-parallel vector field.
\begin{proposition}\label{prop:parallel-vector}
Let $(M^n,J,g)$ be balanced and BTP, and let $Z\in{\mathcal P}$. Then $Z$ is orthogonal to $\im T$, and $L_Z^*$ is a derivation of the pointwise Chern-torsion Lie algebra. Consequently, $H_Z:=\frac12(L_Z+L_Z^*)$ is a $\nab$-parallel Hermitian derivation and $\tr H_Z=0$.
\end{proposition}

\begin{proof}
The statement is trivial for $Z=0$. Suppose that $Z\neq0$.
Since $Z$ has constant length, after rescaling we may assume
$|Z|=1$. Choose a local unitary frame with $e_n=Z$.

We first prove $Z\perp\im T$. In \eqref{eq:curvature-torsion}, put $j=i$ and $k=\ell=n$. The left hand side is zero by Lemma \ref{lem:curvature-parallel-vector}. We get
\begin{equation}\label{eq:first-sum}
0=\sum_s\left\{|T^n_{is}|^2+|T^i_{ns}|^2-|T^s_{in}|^2-T^i_{is}\overline{T^n_{ns}}-T^n_{ns}\overline{T^i_{is}}\right\}.
\end{equation}
Now sum over $i$. The second and third double sums cancel: after interchanging $i$ and $s$, the third one becomes $\sum_{i,s}|T^i_{sn}|^2=\sum_{i,s}|T^i_{ns}|^2$. For the fourth term, first sum over $i$ and use balancedness:
$$
\sum_{i,s}T^i_{is}\overline{T^n_{ns}}=\sum_s\left(\sum_iT^i_{is}\right)\overline{T^n_{ns}}=0.
$$
The fifth term is its complex conjugate. Hence \eqref{eq:first-sum}, summed over $i$, reduces to $\sum_{i,s}|T^n_{is}|^2=0$. Therefore
\begin{equation}\label{eq:no-Z-output}
T^n_{is}=0\qquad\text{for all }i,s.
\end{equation}
This is exactly $Z\perp\im T$.

We next prove that $L_Z^*$ is a derivation. Put $j=n$ in \eqref{eq:curvature-torsion}. Both curvature terms vanish by Lemma \ref{lem:curvature-parallel-vector}. The two terms containing an upper index $n$ vanish by \eqref{eq:no-Z-output}. After multiplying by $-1$, the remaining identity is
\begin{equation}\label{eq:adjoint-derivation}
\sum_s\left\{T^s_{ik}\overline{T^s_{n\ell}}+T^\ell_{ks}\overline{T^i_{ns}}-T^\ell_{is}\overline{T^k_{ns}}\right\}=0.
\end{equation}
The matrix of $L_n=L_Z$ is $(L_n)^a_{\ b}=T^a_{nb}$, so $(L_n^*)^\ell_{\ s}=\overline{T^s_{n\ell}}$. Using the action $\pi$ defined in \eqref{eq:pi-action}, a direct component calculation gives
$$
(\pi(L_n^*)T)^\ell_{ik}=\sum_s\left\{T^s_{ik}\overline{T^s_{n\ell}}+T^\ell_{ks}\overline{T^i_{ns}}-T^\ell_{is}\overline{T^k_{ns}}\right\}.
$$
Thus \eqref{eq:adjoint-derivation} says $\pi(L_n^*)T=0$, namely, $L_n^*$ is a derivation.

Since $L_Z$ is an inner derivation, $H_Z=(L_Z+L_Z^*)/2$ is a Hermitian derivation. Since $L_Z=T(Z,\cdot)$, the Leibniz rule for covariant
derivatives gives
$$
(\nabla^b_XL_Z)(Y)
=
(\nabla^b_XT)(Z,Y)+T(\nabla^b_XZ,Y).
$$
Both terms vanish because $\nabla^bT=0$ and $\nabla^bZ=0$.
Hence $\nabla^bL_Z=0$.

Moreover, since $\nabla^bg=0$, we have $\nabla^b_X(L_Z^*)=(\nabla^b_XL_Z)^*.$ Thus $\nabla^bL_Z^*=0$. Consequently,
$$
\nabla^bH_Z
=
\frac12\bigl(\nabla^bL_Z+\nabla^bL_Z^*\bigr)
=0.
$$
Finally, balancedness gives $\tr L_Z=0$, and therefore $\tr H_Z=\re\,\tr L_Z=0 $.
\end{proof}

\begin{corollary}\label{cor:closed-variations}
For every $Z\in{\mathcal P}$, the real $(1,1)$-form $\alpha_{H_Z}$ is closed and $\nab$-parallel.
\end{corollary}

\begin{proof}
By Proposition \ref{prop:parallel-vector}, $H_Z$ is a Hermitian
endomorphism, a derivation of the pointwise Chern-torsion Lie
algebra, and satisfies $\nabla^bH_Z=0$. Since $H_Z^*=H_Z$ and
$H_ZJ=JH_Z$, the tensor
$\alpha_{H_Z}(X,Y)=g(JH_ZX,Y)$ is a real $(1,1)$-form.

Moreover, the Bismut connection preserves both $g$ and $J$.
Hence, for all vector fields $U,X,Y$,
$$
(\nabla^b_U\alpha_{H_Z})(X,Y)
=
g\bigl(J(\nabla^b_UH_Z)X,Y\bigr)=0.
$$
Thus $\nabla^b\alpha_{H_Z}=0$.

Finally, $H_Z$ is a derivation. Lemma
\ref{lem:closed-derivation} therefore gives
$d\alpha_{H_Z}=0$.
\end{proof}

\vspace{0.3cm}

\section{Holonomy averaging of the pluriclosed metric}

Throughout this section, we assume that $M^n$ is a compact complex manifold, $g$ is a balanced BTP Hermitian metric on $M$, and $h$ is a pluriclosed Hermitian metric on $M$. Let $\omega_h$ denote the fundamental form of $h$. Since $g$ and $h$ are two Hermitian metrics on $T^{1,0}M$, there exists a unique smooth complex-linear endomorphism $G:T^{1,0}M\to T^{1,0}M$ such that $h(U,\bar V)=g(GU,\bar V)$ for all $U,V\in T^{1,0}M$. The Hermitian symmetry and positivity of $h$ imply that $G^*=G$ and
$G>0$, where the adjoint is taken with respect to $g$. Extending $G$ to $T^{0,1}M$ by complex conjugation, and hence to a real endomorphism of $TM$, we have $GJ=JG$. Thus, for real tangent vectors $X,Y$, one has $h(X,Y)=g(GX,Y)$ and
$\omega_h(X,Y)=g(JGX,Y)$. In general, $G$ need not be
$\nabla^b$-parallel. Our goal is to replace $G$ by a positive
$\nabla^b$-parallel Hermitian endomorphism while preserving the integral pairings needed below.

Fix a point $p\in M$ and a $g$-unitary frame
$u_0:\C^n\to T_p^{1,0}M$. Let $g_0$ denote the complex-bilinear extension of the standard Euclidean metric on $\C^n$, with the normalization
$g_0(\epsilon_i,\overline{\epsilon_j})=\delta_{ij}$ for the
standard unitary basis. Thus a frame $u:\C^n\to T_x^{1,0}M$ is unitary precisely when $g(uv,\overline{uw})=g_0(v,\bar w)$ for all $v,w\in\C^n$. Represent the full Bismut holonomy group through $u_0$ by
$$
\mathcal H_0
=
\left\{
u_0^{-1}P_\lambda u_0:
\lambda\text{ is a piecewise smooth loop based at }p
\right\}
\subset U(n),
$$
where $P_\lambda$ denotes $\nabla^b$-parallel transport. By the
reduction theorem of Kobayashi and Nomizu
\cite[Chapter II, Section 7, Theorem 7.1]{KobayashiNomizuI},
the unitary frames which can be joined to $u_0$ by
$\nabla^b$-horizontal curves form a reduced principal subbundle $\mathcal Q$ of the unitary frame bundle, with structure group $\mathcal H_0$. More explicitly, for $x\in M$,
$$
\mathcal Q_x
=
\left\{
P_\gamma\circ u_0:
\gamma\text{ is a piecewise smooth curve from }p\text{ to }x
\right\}.
$$
If $u,u'\in\mathcal Q_x$, then $u'=ub$ for some
$b\in\mathcal H_0$. For the averaging argument, set
$\mathcal H=\overline{\mathcal H_0}^{\,U(n)}$. Since
$\mathcal H$ is a closed subgroup of the compact group $U(n)$,
it is compact and therefore carries a normalized Haar probability measure $d\mu$. Since $\mathcal H$ is compact, the normalized Haar measure is bi-invariant.

For $x\in M$, choose $u\in\mathcal Q_x$. Then $u^{-1}G_xu$ is a positive Hermitian endomorphism of $\C^n$. We let $a\in\mathcal H$ denote the Haar integration variable and
define
\begin{equation}\label{eq:fiber-average}
\widehat G_x
=
\int_{\mathcal H}
a^{-1}u^{-1}G_xua\,d\mu(a).
\end{equation}
This definition is independent of the choice of
$u\in\mathcal Q_x$. Indeed, if $u'=ub$ with
$b\in\mathcal H_0\subset\mathcal H$, then
$$
\int_{\mathcal H}
a^{-1}(u')^{-1}G_xu'a\,d\mu(a)
=
\int_{\mathcal H}
(ba)^{-1}u^{-1}G_xu(ba)\,d\mu(a),
$$
which equals $\widehat G_x$ after the change of variable
$c=ba$ and the left invariance of Haar measure. Since
$\mathcal Q$ admits local smooth sections, formula
\eqref{eq:fiber-average} also shows that
$x\mapsto\widehat G_x$ is smooth. We now define
\begin{equation}\label{eq:base-average}
K_0
=
\frac1{\Vol_g(M)}
\int_M\widehat G_x\,dV_g(x),
\end{equation}
where $\Vol_g(M)=\int_MdV_g$.

\begin{lemma}
\label{lem:average-pairing}
The endomorphism $K_0$ defined in
\eqref{eq:base-average} is positive Hermitian and commutes with
$\mathcal H$. Consequently, it determines a unique global
positive Hermitian endomorphism $K:T^{1,0}M\to T^{1,0}M$ with
$K_p=u_0K_0u_0^{-1}$ and $\nabla^bK=0$. More explicitly, for
$x\in M$ and $u\in\mathcal Q_x$, one has
$$
K_x=uK_0u^{-1}.
$$
If $F$ is any $\nabla^b$-parallel Hermitian endomorphism, then
\begin{equation}\label{eq:average-endomorphism}
\tr(F_pK_p)
=
\frac1{\Vol_g(M)}
\int_M\tr(F_xG_x)\,dV_g(x).
\end{equation}
Equivalently, under the natural trace pairing, if $\Psi$ is
any real $\nabla^b$-parallel $(n-1,n-1)$-form, then
\begin{equation}\label{eq:average-form}
\int_M\Psi\wedge\omega_h
=
\int_M\Psi\wedge\alpha_K,
\end{equation}
where $\alpha_K(X,Y)=g(JKX,Y)$.
\end{lemma}

\begin{proof}
We first prove positivity. For every nonzero $v\in\C^n$, the
unitarity of $u$ and $a\in\mathcal H\subset U(n)$ gives
$$
\begin{aligned}
g_0(\widehat G_xv,\bar v)
&=
\int_{\mathcal H}
g_0\bigl(a^{-1}u^{-1}G_xuav,\bar v\bigr)\,d\mu(a)\\
&=
\int_{\mathcal H}
g\bigl(G_x(uav),\overline{uav}\bigr)\,d\mu(a)>0.
\end{aligned}
$$
The last inequality follows from $G_x>0$. Moreover, every
$a^{-1}u^{-1}G_xua$ is Hermitian, so $\widehat G_x$ is positive
Hermitian. Hence $K_0$ is also positive Hermitian.

We next prove the $\mathcal H$-invariance. For
$b\in\mathcal H$, the right invariance of the normalized Haar
measure gives
$$
b^{-1}\widehat G_xb
=
\int_{\mathcal H}
(ab)^{-1}u^{-1}G_xu(ab)\,d\mu(a)
=
\widehat G_x.
$$
Thus $\widehat G_x$ commutes with every element of $\mathcal H$.
Integrating over $M$ gives $b^{-1}K_0b=K_0$ for every
$b\in\mathcal H$.

We now define the global endomorphism $K$. If
$u,u'\in\mathcal Q_x$, then $u'=ub$ for some
$b\in\mathcal H_0\subset\mathcal H$. Since $K_0$ commutes with
$\mathcal H$, one has
$u'K_0(u')^{-1}=uK_0u^{-1}$. Hence
$K_x=uK_0u^{-1}$ is independent of the choice of
$u\in\mathcal Q_x$ and defines a smooth positive Hermitian
endomorphism of $T^{1,0}M$.

To see that $K$ is $\nabla^b$-parallel, let $\gamma$ be a
piecewise smooth curve from $x$ to $y$. If
$u\in\mathcal Q_x$, then $P_\gamma u\in\mathcal Q_y$, and hence
$$
K_y
=
(P_\gamma u)K_0(P_\gamma u)^{-1}
=
P_\gamma K_xP_\gamma^{-1}.
$$
Thus $K$ is preserved by $\nabla^b$-parallel transport, so
$\nabla^bK=0$. Its uniqueness with the prescribed value
$K_p=u_0K_0u_0^{-1}$ follows from uniqueness of parallel
transport on the connected manifold $M$.

Now let $F$ be a $\nabla^b$-parallel Hermitian endomorphism and put $F_0=u_0^{-1}F_pu_0$. If $u=P_\gamma u_0\in\mathcal Q_x$, then the parallelness of $F$ gives $u^{-1}F_xu=F_0$. Moreover, $F_0$ commutes with $\mathcal H_0$: parallel transport around a loop based at $p$ preserves $F_p$. Since the commutant of $F_0$ in $U(n)$ is closed and $\mathcal H=\overline{\mathcal H_0}^{\,U(n)}$, $F_0$ also commutes with every element of $\mathcal H$. Therefore, for every $a\in\mathcal H$,
$$
\tr\bigl(F_0a^{-1}u^{-1}G_xua\bigr)
=
\tr\bigl(aF_0a^{-1}u^{-1}G_xu\bigr)
=
\tr(F_xG_x).
$$
Averaging over $\mathcal H$ yields
$\tr(F_0\widehat G_x)=\tr(F_xG_x)$. Hence
$$
\tr(F_pK_p)
=
\tr(F_0K_0)
=
\frac1{\Vol_g(M)}
\int_M\tr(F_xG_x)\,dV_g(x),
$$
which proves \eqref{eq:average-endomorphism}.

It remains to prove \eqref{eq:average-form}. At each point
$x\in M$, the trace pairing $(A,E)\mapsto\tr(AE)$ is
nondegenerate on the real vector space of Hermitian
endomorphisms of $T_x^{1,0}M$. Therefore there is a unique
Hermitian endomorphism $F_{\Psi,x}$ such that
$$
\Psi\wedge\alpha_E
=
\tr(F_\Psi E)\,dV_g
$$
for every Hermitian endomorphism $E$. Since $\Psi$ is
$\nabla^b$-parallel and $\nabla^b$ preserves $g$, $J$, and
$dV_g$, the defining identity is preserved by parallel
transport. By uniqueness, $F_\Psi$ is
$\nabla^b$-parallel.

Since $\omega_h=\alpha_G$, we have
$$
\int_M\Psi\wedge\omega_h
=
\int_M\tr(F_\Psi G)\,dV_g.
$$
On the other hand, both $F_\Psi$ and $K$ are
$\nabla^b$-parallel, so $\tr(F_\Psi K)$ is constant on $M$.
Consequently,
$$
\int_M\Psi\wedge\alpha_K
=
\Vol_g(M)\tr(F_{\Psi,p}K_p).
$$
Applying \eqref{eq:average-endomorphism} to $F=F_\Psi$ gives
$$
\Vol_g(M)\tr(F_{\Psi,p}K_p)
=
\int_M\tr(F_{\Psi,x}G_x)\,dV_g(x).
$$
Combining the last three identities proves
\eqref{eq:average-form}.
\end{proof}

\vspace{0.3cm}

\section{Determinant maximization and balancedness}

Throughout this section, we assume that $M^n$ is a compact complex manifold, $g$ is a balanced BTP Hermitian metric on $M$, and $h$ is a pluriclosed Hermitian metric on $M$, $K$ is the positive
$\nab$-parallel Hermitian endomorphism obtained in
Lemma \ref{lem:average-pairing}.
Let
$$
{\mathcal E}=\operatorname{span}_{\R}\{H_Z:Z\in{\mathcal P}\}.
$$
This is a finite-dimensional real vector space because a global parallel tensor is determined by its value at one point. By Proposition \ref{prop:parallel-vector} and Corollary \ref{cor:closed-variations}, every $H\in{\mathcal E}$ is Hermitian, $\nab$-parallel, trace-free, and a derivation; moreover $d\alpha_H=0$.

Consider the open convex set in the affine space $K+\mathcal E$
\begin{equation}\label{eq:cone}
{\mathcal C}=\{D=K+H:H\in{\mathcal E},\ D>0\}.
\end{equation}
Here $D>0$ means that $D$ is positive definite as a Hermitian endomorphism. Since every element of $K+\mathcal E$ is $\nab$-parallel, its trace, determinant, and eigenvalues are constant on the connected manifold $M$. Thus $D\mapsto\det D$ may be regarded as a real-valued function on the finite-dimensional affine space $K+\mathcal E$, while $D\mapsto\log\det D$ is well defined on the positive open subset $\mathcal C$.

\begin{lemma}
\label{lem:det-max}
The function $D\mapsto\log\det D$ attains its maximum on $\mathcal C$ at a point which is interior relative to the affine space $K+\mathcal E$.
\end{lemma}

\begin{proof}
The set $\mathcal C$ is nonempty, since $K\in\mathcal C$.
Moreover, every $D=K+H\in\mathcal C$ satisfies
$\tr D=\tr K$, because $\tr H=0$ for all $H\in\mathcal E$.
The same identity holds on the closure
$\overline{\mathcal C}$ taken in the finite-dimensional affine
space $K+\mathcal E$.

Every $D\in\overline{\mathcal C}$ is positive semidefinite.
Hence, if $\lambda_1,\ldots,\lambda_n$ are its eigenvalues, then
$\lambda_i\geq0$ and
$\sum_i\lambda_i=\tr K$. In particular,
$0\leq\lambda_i\leq\tr K$ for every $i$, so
$\overline{\mathcal C}$ is bounded. The closure $\overline{\mathcal C}$ is closed by definition. Conversely, if $D\in K+\mathcal E$ is positive semidefinite, then $(1-t)D+tK\in\mathcal C$ for every $0<t<1$, and this path converges to $D$ as $t\to0$. Hence
$$
\overline{\mathcal C}=(K+\mathcal E)\cap\{D\geq0\}.
$$
Thus $\overline{\mathcal C}$ is closed and bounded in the finite-dimensional affine space $K+\mathcal E$, and therefore compact.

The continuous function $D\mapsto\det D$ thus attains its maximum
at some $D_0\in\overline{\mathcal C}$. Since $K>0$, one has
$\det K>0$. On the other hand, every point of the relative
boundary of $\mathcal C$ is singular and hence has determinant
zero. Therefore $D_0$ cannot lie on the boundary, and so
$D_0\in\mathcal C$. Since the logarithm is strictly increasing
on $(0,\infty)$, $D_0$ also maximizes $\log\det D$ on
$\mathcal C$. As $\mathcal C$ is open in $K+\mathcal E$, the
point $D_0$ is an interior point relative to this affine space.
\end{proof}

For a positive $\nab$-parallel endomorphism $D$, define the complex-linear $1$-form
\begin{equation}\label{eq:weighted-character}
\chi_D(X)=\tr(D^{-1}L_X).
\end{equation}
The trace is the ordinary complex trace on $\End(T^{1,0}M)$. Because $D$, $D^{-1}$, and $T$ are $\nab$-parallel, $\chi_D$ is also $\nab$-parallel. We write
$\langle X,Y\rangle_g:=g(X,\bar Y)$ for the Hermitian
inner product on $T^{1,0}M$, taken to be complex-linear in the
first variable. By the Hermitian Riesz representation theorem,
there is a unique $(1,0)$-vector field $Z_D$ such that
$\chi_D(X)=\langle X,Z_D\rangle_g
=g(X,\overline{Z_D})$.

For a positive Hermitian endomorphism $D$, denote by $g_D$  the
Hermitian metric given by
$g_D(U,\bar V)=g(DU,\bar V)$ for $U,V\in T^{1,0}M$. Its
fundamental form is
$\alpha_D(X,Y)=g(JDX,Y)$. We denote by $T^D$ the Chern torsion
of $g_D$ and by $\eta^D$ its Gauduchon torsion $1$-form. Throughout
the argument, $T$ continues to denote the Chern torsion of the
original balanced BTP metric $g$.

\begin{lemma}
\label{lem:balanced-character}
Let $D=D^*>0$ be $\nab$-parallel. Then the following conditions
are equivalent:
\begin{enumerate}
\item the Hermitian metric $g_D$ is balanced;
\item the Gauduchon torsion $1$-form $\eta^D$ of $g_D$ vanishes;
\item $\chi_D=0$.
\end{enumerate}
\end{lemma}

\begin{proof}
Fix $p\in M$ and choose a local $g$-unitary frame $e={}^t\!(e_1,\ldots,e_n)$ such that $\theta^b|_p=0$ and $De_i=d_ie_i$, where $d_i>0$. Set $\widetilde e_i=d_i^{-1/2}e_i$ and $\psi_i=d_i^{1/2}\vphi_i$. Then $\{\widetilde e_i\}$ is $g_D$-unitary with dual coframe $\{\psi_i\}$. Comparing \eqref{eq:partial-alpha} for $H=D$ with the expression for $\partial\alpha_D$ in this coframe gives
\begin{equation}\label{eq:TD-components}
(T^D)^k_{ij}=\frac{d_i+d_j-d_k}{\sqrt{d_id_jd_k}}T^k_{ij}.
\end{equation}
Thus, writing $\eta^D=\sum_j\eta^D_j\psi_j$, we have
\begin{equation}\label{eq:etaD-components}
\eta^D_j=\sum_i(T^D)^i_{ij}=\sqrt{d_j}\sum_i\frac1{d_i}T^i_{ij}.
\end{equation}
Since a Hermitian metric is balanced exactly when its Gauduchon torsion $1$-form vanishes, {\rm (i)} and {\rm (ii)} are equivalent.

On the other hand, $\chi_D(e_j)=\tr(D^{-1}L_{e_j})=\sum_i d_i^{-1}T^i_{ji}=-\sum_i d_i^{-1}T^i_{ij}$. Hence
\begin{equation}\label{eq:weighted-balanced}
\eta^D_j=-\sqrt{d_j}\,\chi_D(e_j),
\qquad 1\leq j\leq n.
\end{equation}
Since $d_j>0$ and $\{e_1,\ldots,e_n\}$ is a basis, $\eta^D=0$ if and only if $\chi_D=0$. Thus {\rm (ii)} and {\rm (iii)} are equivalent.
\end{proof}

We now return to the determinant maximizer. The first variation of $\log\det$ in the directions supplied by the parallel derivations $H_Z$ forces $\chi_D$ to vanish.
\begin{proposition}
\label{prop:det-balanced}
Under the assumptions of Theorem \ref{thm:main}, let $D$ be an interior maximizer of $\log\det$ on
$\mathcal C$, whose existence is given by
Lemma \ref{lem:det-max}. Then $\chi_D=0$. Consequently,
$\alpha_D$ is a positive balanced Hermitian form. Moreover,
$\alpha_D-\alpha_K$ is closed.
\end{proposition}

\begin{proof}
Let $Z_D$ be the $(1,0)$-vector field determined by
$\chi_D(X)=g(X,\overline{Z_D})$. Since $\chi_D$ and $g$ are
$\nab$-parallel, $Z_D$ is $\nab$-parallel, and hence
$Z_D\in\mathcal P$. By the definition of $\mathcal E$, we have
$H_{Z_D}\in\mathcal E$.

Since $D>0$ and positivity is an open condition,
$D+tH_{Z_D}\in\mathcal C$ for all sufficiently small
$t\in\mathbb R$. As $D$ is an interior maximum of
$\log\det$, we have
$$
0
=
\frac{d}{dt}\bigg|_{t=0}
\log\det(D+tH_{Z_D})
=
\tr(D^{-1}H_{Z_D}).
$$
Since $(D^{-1})^*=D^{-1}$ and the trace is cyclic, we get
$$
\tr(D^{-1}L_{Z_D}^*)
=
\tr(L_{Z_D}^*D^{-1})
=
\overline{\tr(D^{-1}L_{Z_D})}.
$$
Using $H_{Z_D}=(L_{Z_D}+L_{Z_D}^*)/2$, we obtain
$$
\begin{aligned}
\tr(D^{-1}H_{Z_D})
&=
\frac12\left(
\tr(D^{-1}L_{Z_D})
+
\tr(D^{-1}L_{Z_D}^*)
\right)\\
&=
\re\,\tr(D^{-1}L_{Z_D})
=
\re\,\chi_D(Z_D).
\end{aligned}
$$
By the defining relation for $Z_D$,
$\chi_D(Z_D)=g(Z_D,\overline{Z_D})=|Z_D|_g^2$.
Consequently, $|Z_D|_g^2=0$, and hence $Z_D=0$ and
$\chi_D=0$.

Lemma \ref{lem:balanced-character} now implies that
$\alpha_D$ is balanced. Since $D>0$, it is a positive
Hermitian form. Finally, $D-K\in\mathcal E$. By linearity and Corollary \ref{cor:closed-variations}, every $H\in\mathcal E$ satisfies $d\alpha_H=0$. Therefore
$d\alpha_{D-K}=0$. Since
$\alpha_{D-K}=\alpha_D-\alpha_K$, the last assertion follows.
\end{proof}

\vspace{0.3cm}

\section{Parallel forms and Gauduchon identities}

Throughout this section, we assume that $M^n$ is a compact complex manifold, $g$ is a balanced BTP Hermitian metric on $M$, and $h$ is a pluriclosed Hermitian metric on $M$. We will need to differentiate several $\nabla^b$-parallel forms. The BTP condition implies that $\nabla^b$-parallel forms remain
$\nabla^b$-parallel after applying $d$, $\partial$, or
$\bar\partial$.

\begin{lemma}
\label{lem:parallel-d}
Let $\beta\in\Omega^q(M,\C)$ satisfy $\nab\beta=0$. Then
$d\beta$, $\partial\beta$, and $\bar\partial\beta$ are all
$\nab$-parallel.
\end{lemma}

\begin{proof}
Since $g$ is BTP, we have $\nab\Tb=0$. For real vector fields
$X_0,\ldots,X_q$, the exterior derivative of $\beta$ is given by
\begin{equation}\label{eq:d-bismut-torsion}
\begin{aligned}
d\beta(X_0,\ldots,X_q)
={}&
\sum_{a=0}^q(-1)^a
(\nabla^b_{X_a}\beta)
(X_0,\ldots,\widehat{X_a},\ldots,X_q)\\
&+
\sum_{0\leq a<b\leq q}(-1)^{a+b+1}
\beta\bigl(
\Tb(X_a,X_b),
X_0,\ldots,\widehat{X_a},\ldots,
\widehat{X_b},\ldots,X_q
\bigr),
\end{aligned}
\end{equation}
where $\Tb(X,Y)=\nabla^b_XY-\nabla^b_YX-[X,Y]$.

Since $\nab\beta=0$, the first sum vanishes. Thus $d\beta$ is
obtained algebraically from the two $\nab$-parallel tensors
$\beta$ and $\Tb$ by tensor products, permutations, and
contractions. Hence
$$
\nab(d\beta)=0.
$$

Since $\nab J=0$, the Bismut connection preserves the
bidegree decomposition of complex differential forms. Writing
$\beta=\sum_{r+s=q}\beta^{r,s}$, each $\beta^{r,s}$ is
$\nab$-parallel, and hence so is $d\beta^{r,s}$. As $J$ is
integrable,
$$
d\beta^{r,s}
=
\partial\beta^{r,s}
+
\bar\partial\beta^{r,s}.
$$
The two terms are the $(r+1,s)$ and $(r,s+1)$ components of
$d\beta^{r,s}$, respectively. Since the corresponding type
projections are $\nab$-parallel, both terms are
$\nab$-parallel. Summing over $r,s$ gives
$$
\nab(\partial\beta)=0,
\qquad
\nab(\bar\partial\beta)=0.
$$
This completes the proof of the lemma. \end{proof}

We apply this observation to the balanced form $\alpha_D$ constructed in the previous section. The pluriclosed metric $h$, together with the averaging identity, gives one additional Gauduchon condition for $\alpha_D$.
\begin{proposition}
\label{prop:additional-gauduchon}
Assume $n\geq3$. Let $D$ be the determinant-maximizing endomorphism obtained in
Proposition \ref{prop:det-balanced}, and set
$$
\Phi_D=\sqrt{-1}\,\partial\bar\partial\alpha_D^{n-2}.
$$
Then $\Phi_D$ is a real $\nabla^b$-parallel
$(n-1,n-1)$-form and
$$
\partial\bar\partial\alpha_D^{n-2}\wedge\alpha_D=0.
$$
Equivalently, the Hermitian metric with fundamental form
$\alpha_D$ is $(n-2)$-Gauduchon.
\end{proposition}

\begin{proof}
Since $D$ is $\nab$-parallel and $\nab$ preserves $g$ and
$J$, the form $\alpha_D$ is $\nab$-parallel. Hence
$\alpha_D^{n-2}$ is also $\nab$-parallel. Applying
Lemma \ref{lem:parallel-d} to $\alpha_D^{n-2}$, we obtain that
$\bar\partial\alpha_D^{n-2}$ is $\nab$-parallel. Applying the
same lemma once more to
$\bar\partial\alpha_D^{n-2}$ shows that
$\partial\bar\partial\alpha_D^{n-2}$ is
$\nab$-parallel. Therefore $\Phi_D$ is a
$\nab$-parallel $(n-1,n-1)$-form.

Since $\alpha_D^{n-2}$ is real and
$\partial\bar\partial+\bar\partial\partial=0$, one has
$\overline{\Phi_D}=\Phi_D$. Thus $\Phi_D$ is real.

The auxiliary Hermitian metric $h$ is pluriclosed, so
$\partial\bar\partial\omega_h=0$. By Stokes' theorem, applied
successively to $\partial$ and $\bar\partial$, one has
$$
\int_M\partial\bar\partial\alpha_D^{n-2}\wedge\omega_h
=
\int_M\alpha_D^{n-2}\wedge\partial\bar\partial\omega_h
=
0.
$$
Consequently,
\begin{equation}\label{eq:PhiD-omega-h}
\int_M\Phi_D\wedge\omega_h=0.
\end{equation}

Since $\Phi_D$ is a real $\nabla^b$-parallel
$(n-1,n-1)$-form, the holonomy-averaging identity
\eqref{eq:average-form} gives
\begin{equation}\label{eq:PhiD-alpha-K}
\int_M\Phi_D\wedge\alpha_K=0.
\end{equation}

By Proposition \ref{prop:det-balanced}, the real $(1,1)$-form
$\alpha_D-\alpha_K$ is closed. Since
$$
0=d(\alpha_D-\alpha_K)
=\partial(\alpha_D-\alpha_K)+\bar\partial(\alpha_D-\alpha_K),
$$
and the two summands have bidegrees $(2,1)$ and $(1,2)$,
respectively, they vanish separately. Another integration by
parts therefore gives
\begin{equation}\label{eq:PhiD-difference}
\begin{aligned}
\int_M\Phi_D\wedge(\alpha_D-\alpha_K)
&=
\sqrt{-1}\int_M\partial\bar\partial\alpha_D^{n-2}
\wedge(\alpha_D-\alpha_K)\\
&=
\sqrt{-1}\int_M\alpha_D^{n-2}\wedge
\partial\bar\partial(\alpha_D-\alpha_K)
=0.
\end{aligned}
\end{equation}
Combining \eqref{eq:PhiD-alpha-K} and
\eqref{eq:PhiD-difference}, we find
$$
\int_M\Phi_D\wedge\alpha_D=0.
$$

Finally, both $\Phi_D$ and $\alpha_D$ are
$\nabla^b$-parallel, and hence so is the top-degree form
$\Phi_D\wedge\alpha_D$. Since $\nabla^b$ is a metric connection,
$dV_g$ is $\nabla^b$-parallel. As $M$ is connected, there is a
constant $c$ such that
$\Phi_D\wedge\alpha_D=c\,dV_g$. Integrating over $M$ gives
$c\Vol_g(M)=0$, and therefore $c=0$. Hence
$\Phi_D\wedge\alpha_D=0$ pointwise. Since
$\Phi_D=\sqrt{-1}\,\partial\bar\partial\alpha_D^{n-2}$, this is
equivalent to
$$
\partial\bar\partial\alpha_D^{n-2}\wedge\alpha_D=0.
$$
Thus $\alpha_D$ is $(n-2)$-Gauduchon.
\end{proof}

For the final positivity argument we use the following special case of the Hodge--Riemann bilinear relations. We include the short verification in order to fix the sign under our normalization.
\begin{lemma}[{\cite[Chapter V]{Wells}}]
\label{lem:HR}
Let $V$ be a Hermitian vector space of complex dimension
$n\geq3$, and let $\rho$ be its fundamental form. If $\beta$ is
a primitive $(2,1)$-form, namely $\Lambda_\rho\beta=0$, then
there exists a constant $c_n>0$, depending only on the
normalization of $\rho$, such that
$$
-\sqrt{-1}\,
\beta\wedge\bar\beta\wedge\rho^{n-3}
=
c_n|\beta|_\rho^2\,dV_\rho.
$$
In particular, the left hand side is nonnegative and vanishes
if and only if $\beta=0$.
\end{lemma}

\begin{proof}
This is the primitive $(2,1)$ case of the Hodge--Riemann bilinear
relations; see \cite[Chapter V]{Wells}. Equivalently, the Hodge-star
formula for the primitive $(1,2)$-form $\bar\beta$ has the form
$$
*_\rho\bar\beta=-c'_n\sqrt{-1}\,\bar\beta\wedge\rho^{n-3}
$$
for a constant $c'_n>0$ depending only on the normalization of
$\rho$. Hence
$$
|\beta|_\rho^2dV_\rho
=\beta\wedge *_\rho\bar\beta
=-c'_n\sqrt{-1}\,\beta\wedge\bar\beta\wedge\rho^{n-3}.
$$
Taking $c_n=(c'_n)^{-1}>0$ gives the stated identity. The final
assertion follows immediately.
\end{proof}

\vspace{0.3cm}

\section{Proof of the main theorem}

\begin{proof}[Proof of Theorem \ref{thm:main}]
Let $M^n$ be a compact complex manifold. Suppose that $M$ admits a balanced Bismut torsion-parallel
Hermitian metric $g$ and a pluriclosed Hermitian metric $h$. If $n=1$, every Hermitian metric is K\"ahler. If $n=2$, the balanced condition gives $d\omega=0$, so $g$ is already K\"ahler. We therefore assume $n\geq3$.

Let $K$ be the positive $\nabla^b$-parallel endomorphism obtained in Lemma \ref{lem:average-pairing}, and let $D$ be the determinant-maximizing endomorphism given by Lemma \ref{lem:det-max} and Proposition \ref{prop:det-balanced}. By Proposition \ref{prop:det-balanced}, the form $\alpha_D(X,Y)=g(JDX,Y)$ is positive and balanced, $\nabla^bD=0$, and $\alpha_D-\alpha_K$ is closed.

Since $\alpha_D$ is balanced, $d(\alpha_D^{n-1})=0$, and hence
\begin{equation}\label{eq:main-gauduchon} \partial\bar\partial\alpha_D^{n-1}=0.
\end{equation}
On the other hand, Proposition \ref{prop:additional-gauduchon} gives \begin{equation}\label{eq:main-nminus2} \partial\bar\partial\alpha_D^{n-2}\wedge\alpha_D=0.
\end{equation}
We compare these two identities pointwise.

For any real $(1,1)$-form $\rho$ and any integer $m\geq2$, the
graded Leibniz rule gives
\begin{equation}\label{eq:power-ddbar}
\partial\bar\partial(\rho^m)
=
m\,\partial\bar\partial\rho\wedge\rho^{m-1}
+
m(m-1)\,
\partial\rho\wedge\bar\partial\rho\wedge\rho^{m-2}.
\end{equation}
Indeed,
$\bar\partial(\rho^m)=m\,\bar\partial\rho\wedge\rho^{m-1}$.
Applying $\partial$ gives the first term in
\eqref{eq:power-ddbar} together with
$-m(m-1)\bar\partial\rho\wedge\partial\rho\wedge\rho^{m-2}$.
Since $\partial\rho$ and $\bar\partial\rho$ both have total degree
$3$, interchanging them introduces another minus sign and yields
the second term in \eqref{eq:power-ddbar}. For $m=1$, one simply
has $\partial\bar\partial(\rho^m)=\partial\bar\partial\rho$.

Set
$$
X=\sqrt{-1}\,\partial\bar\partial\alpha_D
   \wedge\alpha_D^{n-2},
\qquad
Y=\sqrt{-1}\,\partial\alpha_D\wedge\bar\partial\alpha_D
   \wedge\alpha_D^{n-3}.
$$
Applying \eqref{eq:power-ddbar} to
$\rho=\alpha_D$ and $m=n-1$, then using
\eqref{eq:main-gauduchon}, gives
\begin{equation}\label{eq:main-XY1}
X+(n-2)Y=0.
\end{equation}

Suppose first that $n\geq4$. Applying
\eqref{eq:power-ddbar} with $m=n-2$, wedging with
$\alpha_D$, and using \eqref{eq:main-nminus2}, we obtain
\begin{equation}\label{eq:main-XY2}
X+(n-3)Y=0.
\end{equation}
If $n=3$, equation \eqref{eq:main-nminus2} is simply
$\partial\bar\partial\alpha_D\wedge\alpha_D=0$, hence $X=0$;
thus \eqref{eq:main-XY2} also holds in this case, since
$n-3=0$. Therefore \eqref{eq:main-XY2} is valid for every
$n\geq3$.

Subtracting \eqref{eq:main-XY2} from
\eqref{eq:main-XY1}, we obtain
\begin{equation}\label{eq:torsion-wedge-zero}
Y
=
\sqrt{-1}\,
\partial\alpha_D\wedge\bar\partial\alpha_D
\wedge\alpha_D^{n-3}
=
0.
\end{equation}

Since $\alpha_D$ is balanced,
$\partial(\alpha_D^{n-1})=0$, and hence
$\partial\alpha_D\wedge\alpha_D^{n-2}=0$. By the Lefschetz
criterion, the $(2,1)$-form $\partial\alpha_D$ is therefore
primitive with respect to $\alpha_D$, equivalently
$\Lambda_{\alpha_D}(\partial\alpha_D)=0$. Applying
Lemma \ref{lem:HR} with $\rho=\alpha_D$ and
$\beta=\partial\alpha_D$, we obtain
$$
-\sqrt{-1}\,
\partial\alpha_D\wedge\bar\partial\alpha_D
\wedge\alpha_D^{n-3}
=
c_n|\partial\alpha_D|_{\alpha_D}^2\,dV_{\alpha_D},
$$
where $c_n>0$. By \eqref{eq:torsion-wedge-zero}, the left hand
side vanishes. Hence
$|\partial\alpha_D|_{\alpha_D}^2=0$, and therefore
$\partial\alpha_D=0$.

Since $\alpha_D$ is real,
$\bar\partial\alpha_D=\overline{\partial\alpha_D}=0$. Thus
$d\alpha_D=0$. The form $\alpha_D$ is positive, real, of type
$(1,1)$, and closed, and hence is a K\"ahler form on $(M,J)$.
This proves the theorem.
\end{proof}

\vspace{0.3cm}

\noindent\textbf{Generative AI disclosure.}
During the development of this work, the authors used ChatGPT $5.6$ Sol to assist with exploratory computations, possible proof directions, and editorial polishing. The authors checked the mathematical proofs and computations, and take full responsibility for the contents of the paper.

\vspace{0.3cm}

\noindent\textbf{Declaration on competing interests.}
All authors declare that there are no competing interests for this paper.


\vspace{0.3cm}

\begin{thebibliography}{99}

\bibitem{AI}
B. Alexandrov and S. Ivanov,
\emph{Vanishing theorems on Hermitian manifolds},
Differential Geom. Appl. \textbf{14} (2001), 251--265.

\bibitem{ArroyoNicolini}
R. M. Arroyo and M. Nicolini,
\emph{SKT structures on nilmanifolds},
Math. Z. \textbf{302} (2022), 1307--1320.

\bibitem{Bismut}
J.-M. Bismut,
\emph{A local index theorem for non-K\"ahler manifolds},
Math. Ann. \textbf{284} (1989), 681--699.

\bibitem{CaoZheng}
K. Cao and F. Zheng,
\emph{Fino--Vezzoni conjecture on Lie algebras with Abelian ideals of codimension two},
Math. Z. \textbf{307} (2024), no. 2, Paper No. 31, 23 pp.

\bibitem{Chiose}
I. Chiose,
\emph{Obstructions to the existence of K\"ahler structures on compact complex manifolds},
Proc. Amer. Math. Soc. \textbf{142} (2014), 3561--3568.

\bibitem{Fei}
T. Fei,
\emph{A construction of non-K\"ahler Calabi--Yau manifolds and new solutions to the Strominger system},
Adv. Math. \textbf{302} (2016), 529--550.

\bibitem{FGV}
A. Fino, G. Grantcharov, and L. Vezzoni,
\emph{Astheno-K\"ahler and balanced structures on fibrations},
Int. Math. Res. Not. IMRN 2019, no. 22, 7093--7117.

\bibitem{FP1}
A. Fino and F. Paradiso,
\emph{Generalized K\"ahler almost abelian Lie groups},
Ann. Mat. Pura Appl. (4) \textbf{200} (2021), 1781--1812.

\bibitem{FP2}
A. Fino and F. Paradiso,
\emph{Hermitian structures on a class of almost nilpotent solvmanifolds},
J. Algebra \textbf{609} (2022), 861--925.

\bibitem{FP3}
A. Fino and F. Paradiso,
\emph{Balanced Hermitian structures on almost abelian Lie algebras},
J. Pure Appl. Algebra \textbf{227} (2023), no. 2, Paper No. 107186.

\bibitem{FV2}
A. Fino and L. Vezzoni,
\emph{Special Hermitian metrics on compact solvmanifolds},
J. Geom. Phys. \textbf{91} (2015), 40--53.

\bibitem{FV16}
A. Fino and L. Vezzoni,
\emph{On the existence of balanced and SKT metrics on nilmanifolds},
Proc. Amer. Math. Soc. \textbf{144} (2016), no. 6, 2455--2459.

\bibitem{FinoVezzoniFlow}
A. Fino and L. Vezzoni,
\emph{A note on the pluriclosed flow on balanced manifolds with $c_1=0$},
arXiv:2606.03176, 2026.

\bibitem{FSwann22}
M. Freibert and A. Swann,
\emph{Two-step solvable SKT shears},
Math. Z. \textbf{299} (2021), no. 3--4, 1703--1739.

\bibitem{FSwann}
M. Freibert and A. Swann,
\emph{Compatibility of balanced and SKT metrics on two-step solvable Lie groups},
Transform. Groups \textbf{30} (2025), 235--265.

\bibitem{FuLiYau}
J. Fu, J. Li, and S.-T. Yau,
\emph{Balanced metrics on non-K\"ahler Calabi--Yau threefolds},
J. Differential Geom. \textbf{90} (2012), no. 1, 81--129.

\bibitem{FuWangWu}
J. Fu, Z. Wang, and D. Wu,
\emph{Semilinear equations, the $\gamma_k$ function, and generalized Gauduchon metrics},
J. Eur. Math. Soc. \textbf{15} (2013), 659--680.

\bibitem{FusiGentili}
E. Fusi and G. Gentili,
\emph{A Levi-type decomposition on two-step solvable Lie algebras with a complex structure},
arXiv:2606.13553, 2026.

\bibitem{Gauduchon84}
P. Gauduchon,
\emph{La $1$-forme de torsion d'une vari\'et\'e hermitienne compacte},
Math. Ann. \textbf{267} (1984), 495--518.

\bibitem{GPodesta}
F. Giusti and F. Podest\`a,
\emph{Real semisimple Lie groups and balanced metrics},
Rev. Mat. Iberoam. \textbf{39} (2023), no. 2, 711--729.

\bibitem{JostYau}
J. Jost and S.-T. Yau,
\emph{A nonlinear elliptic system for maps from Hermitian to Riemannian manifolds and rigidity theorems in Hermitian geometry},
Acta Math. \textbf{170} (1993), 221--254.

\bibitem{KobayashiNomizuI}
S. Kobayashi and K. Nomizu,
\emph{Foundations of Differential Geometry, Vol. I},
Interscience Publishers, New York, 1963.

\bibitem{KwongHomogeneous}
J. Kwong,
\emph{The Fino--Vezzoni conjecture on homogeneous spaces},
arXiv:2608.08665, 2026.

\bibitem{LiZheng}
Y. Li and F. Zheng,
\emph{The Fino--Vezzoni conjecture in Hermitian geometry},
Sci. Sin. Math. \textbf{54} (2024), no. 10, 1603--1614.

\bibitem{NiHolonomy}
L. Ni,
\emph{Holonomy and the Ricci curvature of complex Hermitian manifolds},
J. Geom. Anal. \textbf{35} (2025), no. 1, Paper No. 30.

\bibitem{Otiman}
A. Otiman,
\emph{Special Hermitian metrics on Oeljeklaus--Toma manifolds},
Bull. Lond. Math. Soc. \textbf{54} (2022), no. 2, 655--667.

\bibitem{Podesta}
F. Podest\`a,
\emph{Homogeneous Hermitian manifolds and special metrics},
Transform. Groups \textbf{23} (2018), no. 4, 1129--1147.

\bibitem{Verbitsky}
M. Verbitsky,
\emph{Rational curves and special metrics on twistor spaces},
Geom. Topol. \textbf{18} (2014), no. 2, 897--909.

\bibitem{Wells}
R. O. Wells, Jr.,
\emph{Differential Analysis on Complex Manifolds},
Graduate Texts in Mathematics, Vol. 65, Springer, New York, third edition, 2008.

\bibitem{YZ18Cur}
B. Yang and F. Zheng,
\emph{On curvature tensors of Hermitian manifolds},
Comm. Anal. Geom. \textbf{26} (2018), no. 5, 1195--1222.

\bibitem{ZhaoZ22}
Q. Zhao and F. Zheng,
\emph{Curvature characterization of Hermitian manifolds with Bismut parallel torsion},
arXiv:2407.10497; to appear in Trans. Amer. Math. Soc.

\bibitem{ZhaoZhengPluriclosed}
Q. Zhao and F. Zheng,
\emph{Strominger connection and pluriclosed metrics},
J. Reine Angew. Math. \textbf{796} (2023), 245--267.

\bibitem{ZZThreefold}
Q. Zhao and F. Zheng,
\emph{On balanced Hermitian threefolds with parallel Bismut torsion},
arXiv:2506.15141.

\end{thebibliography}
\end{document}